\documentclass[12pt]{article}
\usepackage[english]{babel}
\usepackage{amsmath}
\usepackage{amssymb}
\usepackage{amsthm}
\usepackage[pdftex]{graphicx}
\usepackage{pdfpages}

\newtheorem{theorem}{Theorem}
\newtheorem{lemma}[theorem]{Lemma}
\newtheorem{corollary}[theorem]{Corollary}

\theoremstyle{definition}
\newtheorem{definition}{Definition}

\theoremstyle{remark}
\newtheorem{remark}{Remark}

\definecolor{light-gray}{gray}{0.5}

\begin{document}

\title{Total order compatible with addition on commutative semigroups}

\author{Askold Khovanskii \thanks{The work was partially supported by the Canadian Grant No. 156833-17. }}
\maketitle

\begin{abstract} In the paper we present a detailed exposition of mainly known results (for example, see \cite{1}). We describe all total orders $\succ$  compatible
with addition on  additive subsemigroup $S$ of finite dimensional spaces
over rational numbers. We provide a necessary and sufficient condition
under which  a finitely generated semigroups $S$ equipped with an order
$\succ$ is a well-ordered set. We also present some auxiliary results on
orders compatible with addition on additive subsemigroups of finite
dimensional spaces over real numbers.

 All arguments in this paper are based on two simple theorems in the geometry
of convex (not necessarily closed) sets.  Proofs of these  theorems are
presented for readers's convenience.

  A first version of this paper was written as a handout for my graduate
course on the theory of Newton--Okounkov bodies.
\end{abstract}

\section{Introduction}

A total order $\succ$ on a commutative semigroup $S$ is compatible with addition if, for any triple $x,y,a\in S$ such that $x\succ y$, the inequality   $x+a\succ y+a$ holds.

We are  interested in all such orders on   subsemigroups of the $n$-dimensional  lattice $\mathbb Z^n\subset \mathbb R^n$. We will completely describe such orders  on additive subsemigroups of the  $n$-dimensional space $\mathbb Q^n$ over the field of rational numbers~$\mathbb Q$.

We are also interested in all such orders on the semigroup $\mathbb Z_{\geq 0}^n\subset \mathbb Z^n$ (consisting of all integral points in $\mathbb R^n$ with nonnegative coordinates), which make $\mathbb Z_{\geq 0}^n$ a well-ordered set.  We will completely describe all such orders on any finitely generated subsemigroup of the additive group of the space $\mathbb Q^n$.

Our arguments use  geometry of convex subsets (not necessarily closed or bounded) in  real affine spaces. We use the classical Caratheodory Theorem, which describes the convex hull $\Delta(A)$  (the smallest convex set containing $A$) of a set $A\subset \mathbb R^n$. We also use   a version of the Separation Theorem which holds for any convex  set  $\Delta\subset \mathbb R^n$ (not necessarily closed or bounded) and for any  point $a$ in its complement $a\in \mathbb R^n\setminus \Delta$. For readers's convenience,  we  present proofs of both these theorems in convex geometry.
\medskip

I would like to thank a student Joe Glasheen from  my graduate course on the theory of Newton--Okounkov bodies, who edited English in this paper.

\section{Order compatible with addition on general commutative semigroups}

We are mainly interested  in the  class of additive subsemigroups of real vector spaces.

The following (obvious) Lemma on general commutative semigroups automatically holds for semigroups belonging to this class.

\begin{lemma}\label{lemma1} If commutative semigroup $S$ has a total order compatible with addition, then $S$ has a cancelation property; and an identity $nx=ny$ implies $x=y$, where $x,y\in S$ and $n$ is a natural number.
\end{lemma}

\begin{proof} If $x\prec y$ or $y\succ x$, then, for any $a\in S$, we correspondingly have that $x+a\succ y+a$ or $y+a\succ x+a$. So, if $x+a=y+a$, then $x=y$. Thus, the semigroup $S$ has the  cancelation property.

If $x\succ y$ or $y\succ x$, then  we correspondingly have that $nx\succ ny$ or $ny\succ nx$. So, if $nx=ny$, then $x=y$.
\end{proof}

\begin{corollary}\label{col1} If $S$ satisfies assumption of Lemma \ref{lemma1}, then $S$ can be naturally embedded  to its Grothendieck group $G$; and the group $G$ is a free commutative group (i.e. $G$ has no torsion).
\end{corollary}

\begin{corollary}\label{col2}  Any total order  compatible with addition  on a commutative semigroup $S$ can be uniquely extended to  the total order compatible with addition  on the  Grothendieck group $G$ of the semigroup $S$.
\end{corollary}

\begin{proof} Any elements  $a_1,a_2\in G$ can be  represented  in the form $x_1-y_1=a_1$, $x_2-y_2=a_2$. We  say that $a_1$ is bigger (or correspondingly,  smaller) than $a_2$ if $x_1+y_2$ is bigger  (or correspondingly, smaller) than $x_2+y_1$. The above  order  on $G$ is well defined (i.e. is independent of representations of $a_1,a_2$ as the difference of elements from $S$) and is the only possible extension of the order on $S$ to an order on $G$.
\end{proof}

Thus, a description of all total orders compatible with addition  on a commutative semigroup $S$ is reduced
to a description of all total orders  compatible with addition  on its Grothendieck group $G$.

On any free commutative group $G$ there is a total order compatible with addition. We will construct such order later (see Lemma \ref{lemma9}), when we will discuss lexicographic  orders on real vector spaces. So, the conditions on semigroup $S$ from Lemma \ref{lemma1} are not only necessary but also sufficient for existence of a total order   on $S$ compatible with addition.

\medskip

One can easily check the following two lemmas.

\begin{lemma}\label{set G_+ 1} For any total order  $\succ$ compatible with addition  on a commutative group $G$ the
set $G_+\subset G$ defined by condition $x\in G_+ \Leftrightarrow x\succ
0$ has the following properties:
\begin{enumerate}

\item the set $G_+$ is a semigroup with respect to addition;

\item zero is not in $G_+$;

 \item for any $x\neq 0$ exactly one element from the couple $(x,-x)$ belongs~to~$G_+$.
\end{enumerate}
\end{lemma}

\begin{lemma} \label{set G_+ 2} If a subset $G_+ \subset G $ satisfies the conditions 1)--3)
from the previous lemma \ref{set G_+ 1} then the relation $x\succ y \Leftrightarrow
x-y\in G_+$ defines a total order on the group $G$ compatible with
addition.
\end{lemma}

Let us prove a simple general lemma on well-ordered commutative semigroups, assuming that the ordering is compatible with addition.

\begin{lemma}\label{well-ordered 1} If a commutative semigroup $S$ equipped with a total order $\succ$ compatible with addition is a well-ordered set, then, for any nonzero element $a\in S$, the condition $2a \succ a$ holds. (If $S$ contains the origin, then this condition means that the origin is the smallest element in $S$.)
\end{lemma}

\begin{proof} If for some $a\in S$ the condition $a \succ 2a$ holds, then  the sequence $a, 2a,\dots, na,\dots$ is strictly decreasing and does not contain a smallest element.
\end{proof}

\section{Lexicographic orders}

Lexicographic orders on finite dimensional real vector spaces are very important for us. Let us start with a formal definition, which works even for infinite dimensional real vector spaces.
\medskip

Consider a real vector space  $L$. Let $\{e_{\lambda}\},\lambda\in \Lambda$  be any basis in $L$, where $\Lambda$ is an index set. Choose any well-order on the
set $\Lambda$.

\begin{remark} Note that there are well-orders  on any set $\Lambda$. If the set $\Lambda$ is infinite such order could be very exotic; but, on set containing $n<\infty$ elements, all orders are in one-to-one correspondence with  all  enumerations of  elements in the set $\Lambda$ by indices $1\leq i\leq n$.
\end{remark}
 \medskip

 Using the chosen  well-order on $\Lambda$, one can define a total order  on $L$
compatible with addition. Each vector $v\in L$ has a unique representation  of the form
$v=\sum x_\lambda(v) e_\lambda$, where only finitely many coefficients $x_\lambda(v)$
are not equal to zero.

\begin{definition} Let $a=\sum x_\lambda(a)e_\lambda$ and $b=\sum
x_\lambda(b)e_\lambda$ be two vectors in the space $L$.
Let $\Lambda_{a,b}\subset \Lambda$ be the set of indices such that $x_\lambda(a)\neq
x_\lambda(b)$.  Denote by $\lambda_0$ the smallest
element in $\Lambda_{a,b}$. We say that $a$ is {\sl bigger} than $b$ in the lexicographic order associated with the well-ordered basis $\{e_\lambda\}$ of $L$  if $x_{\lambda_0}(a)>
x_{\lambda_0}(b)$.
\end{definition}
\medskip

\begin{definition} An order $\succ$ on a real vector space  $L$ is {\sl compatible} with multiplication on positive numbers if, for any $x,y \in L$ and any $\mu>0$ such that $x \succ y$, the relation $\mu x \succ\mu y$ holds.
\end{definition}
\medskip

The following Lemma is obvious:

\begin{lemma} Any lexicographic order   on a real vector space $L$ is compatible with addition and with  multiplication by positive numbers.
\end{lemma}
\medskip

\begin{lemma}\label{lemma8} For any total order $\succ$ compatible with addition and multiplication by positive numbers  on a real vector space $L$ the set $L_{+}$, defined by the condition $x\in  L_{+}\Leftrightarrow x\succ 0$,
is convex. In particular this condition holds for any lexicographic order on $L$.
\end{lemma}

\begin{proof} Since the order $\succ$ is compatible with addition and with  multiplication by positive numbers, for any two points $x,y \in L_+ $ and real number $0\leq \lambda\leq 1$, the set $L_+$ contains the point $\lambda x+(1-\lambda y)$.
\end{proof}

\begin{lemma}\label{lemma9} On any free commutative  group $G$, there is a total order compatible
with addition.
\end{lemma}

\begin{proof}  Any free commutative  group $G$ can be naturally embedded in the real vector space $L= G \otimes_{\mathbb Z} \mathbb R$. The lexicographic order on $L$  induces a total order
compatible with addition on any subgroup of $L$.
\end{proof}
\medskip

Thus we see that a total order compatible with addition on a commutative
semigroup $S$ exists if and only if $S$ satisfies the assumptions of Lemma
\ref{lemma1}.
\medskip

From now on, we will deal only with the lexicographic
order on real finite-dimensional vector spaces.

\begin{definition} We will call the lexicographic order associated with the  ordered basis of an $n$-dimensional real vector space $L$ the {\sl lexicographic order related to the coordinate system} $\mathbf x=(x_1,\dots,x_n)$, defined by the basis $e_1,\dots,e_n$. We will denote this order by the symbol $\succ_\mathbf x$.
\end{definition}

\medskip

Let us discuss the geometrical meaning of the  order  $\succ_\mathbf x$ on $L$ related to a coordinate system $\mathbf x=(x_1,\dots,x_n)$.
\medskip

\begin{definition} With the coordinate system $\mathbf x$  one associates a {\sl flag of subspaces} $L=L_0\supset L_1\supset \dots \supset L_n=0,$
where $L_1$ is defined by  the equation $x_1=0$; $L_2$ by the equations $x_1=x_2=0$; so on, up to $L_n$, defined by the equations $x_1=\dots=x_n=0$ (i.e. $L_n=0$).
\end{definition}
\medskip

With the coordinate system  $\mathbf x$, one associates  the collection of open half spaces $L_i^+\subset L_i$ of the space $L_i$;
with the boundary $L_{i+1}$ specified by the condition that $x_{i+1}>0$ on $L_i^+$
\medskip

The flag $L=L_0\supset L_1\supset \dots\supset L_n=0$, together with the collection of chosen half spaces $L_0^+,\dots L_{n-1}^+$, totally determines the lexicographic order~on~$L$.

\begin{definition} The set $X_{+} =\cup _{0\leq i <n}L^+_i$ we will the call {\sl $\mathbf x$-half space related to the coordinate system }$\mathbf x=(x_1,\dots,x_n)$.
\end{definition}
\medskip

The following two Lemmas are obvious.

\begin{lemma}  The $\mathbf x$-half space  $X_+$ of $L$  totally determines the lexicographic order on $L$ related to the coordinate system $\mathbf x=(x_1,\dots,x_n)$. Moreover the the identity $X_+=L_+(\mathbf x)$ holds, where $L_+(\mathbf x)$ is the set of points $a\in L$ such that $a\succ_\mathbf x 0$.
\end{lemma}

\begin{lemma} The orders $\succ_\mathbf x $ and $\succ_\mathbf y $  related to coordinate systems $\mathbf x=(x_1,\dots,x_n)$ and $\mathbf y=(y_1,\dots,y_n)$ on $L$ coincide if and only if the linear map $A:L\to L$ which transforms  the coordinate system $\mathbf x$  to the coordinate system $\mathbf y$  is given by an upper triangular matrix having positive entries on its main diagonal.
\end{lemma}

\medskip

 For any coordinate system $\mathbf x$  on $L$ let $L_{-}(\mathbf x)$ be the set  defined by the following condition: $x\in L_{-}(\mathbf x)\Leftrightarrow -x\in L_{+}(\mathbf x)$ (where $L_+(\mathbf x)$ is the $\mathbf x$-half space of $L$).

\begin{lemma} \label{property of order 1} For any coordinate system $\mathbf x$  on $L$  the sets $L_+(\mathbf x)$ and $L_-(\mathbf x)$ are convex. These sets satisfy the following conditions:

$$L_{+}(\mathbf x)\cup L_{-}(\mathbf x)= L\setminus \{0\},$$
$$L_{+}(\mathbf x)\cap L_{-}(\mathbf x)=\emptyset.$$
\end{lemma}
\medskip

Let us formulate two  Theorems  \ref{geometry of order 1}, \ref{geometry of order 2} which  geometrically characterize lexicographic orders on $L$ without using coordinate systems. We will prove these theorems in the section \ref{subsec6.1}
\medskip

\begin{theorem}\label{geometry of order 1} Let $\succ $ be a total order on a real  finite-dimensional vector space $L$ that is compatible with addition and with multiplication by positive numbers. Then $\succ$ is the lexicographic order $\succ_\mathbf x $ related to some coordinate system $\mathbf x$ on $L$.
\end{theorem}
\medskip

\begin{remark} On a real vector space $L$ of any dimension $n>0$ there are a many total orders compatible with addition which are not lexicographic orders related to some coordinate system (over real numbers). Indeed, one can consider $L$ as a free commutate group with respect to addition and use an (exotic) order on it such as that described in Lemma \ref{lemma9}
\end{remark}
\medskip

\begin{theorem}\label{geometry of order 2}  Let $X_+\subset L$  be a convex set and let $X_{-}$ be the set defined by condition  $x\in X_{-}\Leftrightarrow -x\in L_{+}$. Assume that the following conditions hold:
$$X_{+}\cup X_{-}= L\setminus \{0\},$$
$$X_{+}\cap X_{-}=\emptyset.$$
 Then there is a (unique) coordinate system $\mathbf x$ on $L$, such that $X_{+}=L_{+}(\mathbf x)$, where $L_{+}(\mathbf x)$ is the $\mathbf x$-half space of $L$.
 \end{theorem}

\section{Orders  on subgroups and  subsemigroups of real numbers}

First, we will consider results for $n=1$. Let $G\subset \mathbb R$ be an additive  subgroup of $\mathbb R$ equipped with some total order $\succ$ compatible with addition. As above,  we  denote by $G_+$a semigroup $G_+\subset G$ containing all points $a\in G$ such that $a \succ 0$.

\begin{lemma}\label{convex hull 1} If the convex hull $\Delta(G_+)$ of the  set $G_+$ does not contain the origin, then the order $ \succ$ is either induced by the natural order on the line of real numbers, or by the opposite order, i.e. $a \succ b$ if and only if $a<b$.
\end{lemma}

\begin{proof} The semigroup $G_+$ can contain only positive points, or only negative points. Indeed, if $a>0$ and $b<0$ are contained in $G_+$ then $\Delta (G_+)$ contains the origin, which is not possible. If $G_+$ contains only positive points, then the intersection of the open ray $x>0$ with $G$ is equal to $G_+$. Indeed, assume that   a point $a \in G$  does not belong to $G_+$, but  belongs to the ray. Then, the negative point $-a$ must belong to $G_+$, which contradicts our assumption. Thus,  $a \succ b$ if and only if $a-b>0$.
\medskip

If $G_+$  belongs to the negative ray $x<0$, then similar arguments show that $a \succ b$ if and only if $a-b<0$.
\end{proof}

\begin{lemma}\label{rational group 1} If a semigroup $S\subset \mathbb Q\subset \mathbb R$ contains  only rational points, then the convex hull $\Delta(S)$ contains the origin if and only if the origin belongs~to~$S$.
\end{lemma}

\begin{proof} If $0\in \Delta(S)$, then $S$ contains some positive point $\lambda>0$ and some negative point $\mu<0$. The ratio $\dfrac{\lambda}{\mu}$ is a negative rational number, so $\dfrac{\lambda}{\mu}=-\dfrac {p}{q}$ where $p$ and $q$ are natural numbers. We have that  $q \lambda +\mu p=0$, which means that the semigroup $S$ with the points $\lambda, \mu$ contains the origin.
\end{proof}

\begin{theorem} Let $G\subset \mathbb Q\subset \mathbb R$
be an additive group which contains only points
 with rational coordinates. Then there are exactly
 two total orders on $G$ compatible with addition:
 the natural order ($a \succ b$ if $a>b$) and the reverse order ($a \succ b$ if $a<b$).
 \end{theorem}

 \begin{proof} The semigroup $G_+\subset G$ related to the order $\succ$ cannot contain the origin. Thus, by Lemma \ref{rational group 1} $\Delta (G)$ does not contain the origin either. Thus,  the required statement follows from Lemma \ref{convex hull 1}.
 \end{proof}
 \medskip

In the section \ref{Rn} we  generalize Theorem to the multidimensional case. The multidimensional statement analogous to Lemma \ref{convex hull 1}  is based on a version of the Separation Theorem for convex sets (see Theorem \ref{thSepar}). The multidimensional statement analogous to Lemma \ref{rational group 1}  is based on Caratheodory's theorem (see Theorem \ref{cara-th}).

\begin{lemma}\label{lemma14} Let $S\subset \mathbb R$ be a finitely generated semigroup which contains  only nonnegative numbers. Then, $S$ with the natural  order  induced by  $\mathbb R$ is a well-ordered set. Moreover, for any $l\in \mathbb R$, there are only finitely many elements in $S$ which are smaller than $l$.
\end{lemma}

\begin{proof} Indeed, assume that $C$ is the  smallest nonzero number among generators of $S$. Then, on  any segment $0\leq x\leq l$, there are at most $\frac {l}{C}+1$ elements of the semigroup $S$. So, any subset of $S$ contains a smallest element.
\end{proof}

\begin{theorem} A finitely generated semigroup $\subset \mathbb R$   which contains only rational points, i.e. $S\subset \mathbb Q$; and, which as a total order $\succ$  compatible with addition; is a well-ordered set if and only, for every nonzero element $a\in S$, the inequality $2a\succ a$ holds.
\end{theorem}

\begin{proof} In one direction, the Theorem follows from Lemma \ref{well-ordered 1}.
\medskip

 Let us prove it in the opposite direction.

 Since $S$ contains rational  only points and the order $\succ$ is compatible with addition, the order $\succ$  is induced either by the natural order on $\mathbb R$, or by the opposite  order on $\mathbb R$. By the condition in the Theorem, either $S$ belongs to the ray of nonnegative numbers (if $\succ$ is induced from the natural order); or $S$ belongs to the ray of nonpositive  numbers (if $\succ$ is induced from the opposite order). In both cases,  Lemma \ref{lemma14} implies that $S$ is a well-ordered set.
\end{proof}

\section{Two theorems in the geometry of convex sets}

In this section, we discuss two theorems from convex geometry which we will use later.

\subsection{Version of the Separation Theorem for non-necessarily  closed convex sets}

Let us recall the classical Separation Theorem for closed convex sets in a real finite dimensional space $L$ (for example, see \cite{2}).

\begin{theorem}[Separation Theorem] For any closed convex set $\Delta\subset L$ and for any point $a\in L$ not belonging to $\Delta$, there is a linear function $x:L\to \mathbb R$ such that for any point $b \in \Delta$ the inequality $x(a)<x(b)$ holds.
\end{theorem}

\begin{proof} Choose any Euclidean metric on $L$. Denote by $\rho(v_1,v_2)$ the distance between points $v_1,v_2\in L$. Denote by $f:L\to \mathbb R$ the function whose value at a point $y\in L$ is equal to $\rho(a,y)$. The function $F$ is smooth on  $L\setminus \{ a\}$ and it tends to infinity as $y$ tends to infinity.
\medskip

 Since $\Delta$ is closed, the function $f$ attains its minimum on $\Delta$ at some point $b\in \Delta$.
\medskip

Let $x$ be the linear function on $L$ defined by relation $x(y)=\langle y, b-a\rangle$, where $\langle v_1,v_2\rangle$ is the inner product of the vectors $v_1,v_2\in L$.
\medskip

The gradient $\nabla F_b$ of the function $F$ at the point $b$ is equal to $\dfrac{ b-a}{|b-a|}$. For any point $y\in \Delta$, the segment joining $b$ and $y$ belongs to $\Delta$. Since $f$ attains its minimum on $\Delta$ at the point $b$, the inner product $\langle \nabla F_b, c-b\rangle$ is a nonnegative number.
\medskip

The inequality  $\langle \nabla F_b, c-b\rangle \geq 0$   means that the set $\Delta$ belongs to the closed half space where the function $x$ is bigger than or equal to $x(b)$; while the point $a$ is located in the  open half space where $x$ is smaller than $x(b)$.
\end{proof}

With any system of coordinates $\mathbf x=(x_1,\dots,x_n)$ on $L$ one associates the lexicographic order $\succ_\mathbf x$ on $L$.

\begin{definition} For any point $a\in L$, denote by $L_+(a,\mathbf x)$ the set of points $y\in L$ satisfying the inequality $y\succ_\mathbf x a $. We will call the set $L_a(\mathbf x)$ the {\sl $x$-half space with the vertex $a$}.
\end{definition}

\medskip

The set $L_+(a,\mathbf x)$ is equal to the $x$-half space  $L_{+}(\mathbf x)$  shifted by vector $a$.

\begin{theorem}[Version of the Separation Theorem, see see \cite{1}]\label{thSepar} Let $\Delta\subset L$ be a (not-necessarily  closed) convex set,  and let $a\in L\setminus \Delta$ be a point not belonging to $\Delta$. Then, there is a coordinate system $\mathbf x=(x_1,\dots,x_n)$ in $L$,  such that $\Delta$ belongs to the $x$-half space with vertex $a$, i.e.  $\Delta\subset L_+(a,\mathbf x)$.
\end{theorem}

 \begin{proof} Let $\overline \Delta$ be the closure of $\Delta$. If  $a$ does not belong to $\overline \Delta$, then by the Separation Theorem for closed  convex sets, there is a linear function $x:L\to \mathbb R$ such that, for any $b\in \overline \Delta$, the inequality $x(a)<x(b)$ holds. Let us choose an arbitrary system of coordinates  $\mathbf x=(x_1,x_2,\dots,x_n)$, with $x_1=x$.Then,
 $$L_+(a,\mathbf x)\supset \Delta.$$
  So, for the case under consideration the Theorem is proven.
  \medskip

Now, assume  that $a\in \overline \Delta$. Since $a$ is not in $\Delta$, there is a support hyperplane $LH$ for $\overline \Delta$ at the point $a$.  In relation to $H$, one can consider a linear function $x:L\to \mathbb R$, such that $x$ restricted to $H$ is a constant and $x$ attains it's minimum on $\overline \Delta$ at the point $a$.
\medskip

Consider a convex set $\Delta_1=H\cap \Delta$ in the affine space $H$. If $0\in H$, then $H$  is a linear space of dimension $(n-1)$.
\medskip

If $H$ does not contain the origin, choose any point $O_1\in H$, and consider $H$ a linear space with  origin $O_1$.
\medskip

Note that any linear function $\tilde x$ which is defined on $H$ can be extended to a linear function $x$ on $L$ (if $0\in H$ there is a one parameter family of such extensions; if the origin is not in $H$ such n extension is unique).

\medskip

By induction, we can assume that the Theorem is proven for all  $(n-1)$-dimensional spaces. So, there is a coordinate system $\tilde {\mathbf x}=(\tilde x_2,\dots,\tilde x_n)$ on $H$ such that the set $H_+(a,\tilde {\mathbf x})\subset H$ contains the set $\Delta_1$.
\smallskip

To complete the proof, one can extend the functions $\tilde x_2,\dots,\tilde x_n$ (which are defined on $H$) to linear functions $x_2,\dots,x_n$ on $L$. Consider the coordinate system $\mathbf x=x_1,x_2,\dots,x_n$, with $x_1=x$, where the function $x$ is defined in the first step of our inductive proof.
\medskip

It is easy to check that for the coordinate system $\mathbf x=x_1,x_2,\dots,x_n$ the $x$-half space  $L_+(a,\mathbf x)$ with vertex $a$  contains the set $\Delta$.
\end{proof}

\medskip

Now, we are ready to prove Theorem \ref{geometry of order 2}.
\medskip

 \begin{proof}[Proof of Theorem \ref{geometry of order 2}] By assumption the origin does not belong to the convex set  set $L_+$. So, by our version of Separation Theorem,  there is a coordinate system $\mathbf x$ on $L$ such that the $x$-half space $L_{+} (\mathbf x)$ contains the set $L_+$. These assumptions imply that the sets $L_{+} (\mathbf x)$ and  $L_+$ are equal. Indeed, if there is a point $x\in L_{+} (\mathbf x)$ which is not in $L_+$, then $x$ must belong to the set $-L_{-}$. This means that $-x\in L_+$. We obtain a contradiction since the point $-x$ is not in set
$L_{+} (\mathbf x)$. This contradiction proves the Theorem.
\end{proof}

\medskip

The version of Separation Theorem implies the following corollary:

\begin{corollary} A subset $\Delta$ in a real $n$-dimensional space $L$ is convex, if and only  if it is equal to the  intersection (over all choices of coordinate systems $\mathbf x$  and  points $a\in L$) of all  sets $L_+(a,\mathbf x)$ containing $\Delta$.
\end{corollary}

\begin{proof}
 All sets  $L_+(a,\mathbf x)$ are convex. For any point $a$ not in $\Delta$ there is a set $L_+(a,\mathbf x)$ which contains $\delta$.
\end{proof}

\subsection{Convex geometry related to Caratheodory's theorem}

Let $L$ be any real vector space (perhaps  of infinite dimension).  For a set  $A\subset L$ let us denote by $\Delta(A)$  the convex hull of $A$ (which is not necessarily is closed).

\begin{lemma}\label{convex hull}  A point $x\in L$ belongs to the convex hull  $\Delta(A)$ of a set $A\subset L$ if and only if $x$ belongs to the convex hull of some finite  subset $B$ of the set~$A$.
\end{lemma}

\begin{proof} On the one hand the convex set $\Delta (A)$ must contain the convex hull of each finite set $B\subset A$. On the other hand if $x_1,x_2$ belong to the convex hulls of finite sets $B_1,B_2$, then the segment joining $x_1$ and $x_2$ belongs to the convex hull of the finite  set $B_1\cup B_2$.
\end{proof}

\begin{definition} A set $B \subset L$ containing $k+1$-points  is {\sl affinely independent} if it does not belong to any  affine subspace $L_B\subset L$, with $\dim L_B<k$.
\end{definition}
 \smallskip

Caratheodory's Theorem (for example, see \cite{2}) improves  Lemma \ref{convex hull}.
 \medskip

 \begin{theorem}[Caratheodory Theorem]\label{cara-th} A point $x\in L$ belongs to the convex hull  $\Delta(A)$ of a set $A\subset L$ if and only if $x$ belongs to the convex hull of some finite subset $B\subset A$ which is affinely  independent. Any point $x$ in the smallest affine space containing the set $B$ has a unique representation of the form $x=\sum \lambda_i b_i$, where $b_i$ are points of the set $B$ and $\lambda$ are real numbers such that $\sum \lambda_i=1$.
 \end{theorem}
 \medskip

Caratheodory's Theorem  follows from the geometric Lemma \ref{map} (see below).

\medskip

Let $L_1$ and $L_2$ be affine spaces of dimensions $n$ and $k$, respectivly. Let $\Delta\subset L_1$ be a convex $n$-dimensional polyhedron.
\medskip

\begin{lemma}\label{map}  Let $A:L_1\to L_2$ be an affine map. Then the image $A(\Delta)\subset L_2$ is equal to the union $\cup A(\Gamma_i)$  of the images $A(\Gamma_i)$ of all faces $\Gamma_i$ of $\Delta$  such that $\dim \Gamma_i\leq k$.
\end{lemma}

\begin{proof} Let $a\in A(\Delta)$ be any point in the image of $\Delta$. Its preimage $ L_3= A^{-1}(a)\subset L_1$ is an $(n-k)$-dimensional affine subspace of $L_1$ which intersects the polyhedron $\Delta$. Let $\Gamma$ the lowest dimensional face of $\Delta$ which has  non empty  intersection with $L_3$. Thus, $L_3$ cannot intersect the boundary of $\Gamma$. This condition implies  that $L_3\cap \Gamma$ is an interior point of $\Gamma$. Thus $\dim \Gamma\leq k$.
\end{proof}

\medskip

We will need one more observation from linear algebra.
\smallskip

 Let $\mathbb R^{k+1}$ be the standard linear space  with the standard  basis $e_1,\dots,e_{k+1}$  and standard  the coordinates $\lambda_1,\dots, \lambda_{k+1}$ .
 \smallskip

Let $\mathcal L_k$ be the $k$-dimensional hyperplane  in $R^{k+1}$ defined by the equation
$$\lambda_1+\dots +\lambda_{k+1}=1.$$

\medskip

Each point  $x$ in $\mathcal L_k$  has a unique representation of the form $\sum \lambda_1 e_i$ where $\sum \lambda_1=1$.
\smallskip

Let $\Delta_k$ be the standard simplex in $\mathcal L_k\subset \mathbb R^{k+1}$,  defined by the inequalities $\lambda_1\geq 0,\dots,\lambda_{k+1}\geq 0$. The vertices  of the polyhedron $\Delta_k$  are the endpoints $P_1,\dots, P_{k+1}$ of the vectors $e_1,\dots,e_{+k+1}$.

\medskip

\begin{proof}[Proof of Caratheodory's Theorem] For a point $x\in \Delta(A)$, choose a set $B\subset A$ having the smallest number  $k+1$ of elements $b_1,\dots,b_{k+1}$, such that $x$ belongs to the convex hull $\Delta(b)$ of $B$. Let $L_B$ be the the smallest affine space containing $\Delta(B)$; so $\dim \Delta(B)=\dim L_B\leq k$.
\medskip

Let us show that $\dim L_B=k$. Consider an affine  map $A:\mathcal L_k\to L_B$ which maps the vertices $p_1,\dots,p_{k+1}$ to the points $b_1,\dots, b_{k+1}$ (thus $A(p_i)=b_i)$.
\medskip

 The image $A(\Delta_k)$ is the union of images $A(\Gamma_i)$ of the faces $\Gamma_i$ of $k$-dimensional simplex  $\Delta_k$, such that $\dim \Gamma_i=\dim L_B$. Since $x\in A(\Delta_k)$ and $x$ does not belong to an image of a proper face $\Gamma_i\subset \Delta_k$ one concludes that $\dim L_b=k$ and the map $A:\mathcal L_k\to L_B$ is one-to-one affine map. So the set $B$ is affinely independent, and each point  $y\in L_B$ has a unique representation of the form $y=\sum \lambda_i a_i$, $\sum \lambda_i=1$.
\end{proof}

\section{Total orders compatible with addition on an additive subgroup on $\mathbb R^n$ and $\mathbb Q^n$}\label{Rn}

In this section we will generalize to the multidimensional case the one-dimensional results which we presented earlier. We will use as our main tools the  two theorems from convex geometry presented in the previous section.

\subsection{Lexicographic orders and orders compatible  with addition on  subgroups of $\mathbb R^n$}\label{subsec6.1}

Any  total order compatible with addition on a commutative group $G$  is determined by a semigroup $G_+\subset G$ (see  Lemma \ref{set G_+ 2}).
\smallskip

In this subsection we will consider subgroups $G$ of the additive  group of a real finite dimensional real vector space $L$  equipped
with the order defined by a subsemigroup  $G_+\subset G\subset  L$. We are interested  in the following question:
\smallskip

 Under what conditions on the semigroup
$G_+$ is the total order on $G$  induced by a lexicographic order
on $L$ which is related to some coordinates system $\mathbf x$ on $L$?

\medskip

The following Theorem provides the answer.

\begin{theorem}\label{convex hull and order} The total order on a group $G\subset L$ which is defined by semigroup $G_+\subset G$ is induced by the
order $\succ_\mathbf x$ on $L$ related to some coordinate system $\mathbf x$   if and only if the convex hull $\Delta(G_+)$ of the set $G_+$ does not contain the origin.
\end{theorem}

\begin{proof} Assume that the order  on $G$ is induced by the lexicographic
 order $\succ_\mathbf x$ on $L$. Then, the set $G_+$ is contained in the convex set $L_{+}(\mathbf x)$, which does  not contain the origin. So, the convex hull $\Delta(G_+)$ does not contain  the origin either. This proves the Theorem in one direction.
\medskip

If the origin  does not belong to $\Delta(G_+)$, then, by our  version of the Separation Theorem, there
is a system of coordinates $\mathbf x$ such that $G_+\subset L_{+}(\mathbf x)$.

\medskip

Let us show that $G\cap  L_{+}(\mathbf x)=G_x$. Indeed, if there is a point $a\in  L_{+}(\mathbf x)\cap G$ not belonging to $G_+$, then $-a\in G_+$. But $-a$ does not belong to $L_{+}(\mathbf x)$. We obtain a contradiction which proves the needed statement.
\medskip

This identity $G\cap  L_{+}(\mathbf x)=G_+$  shows that the order on $G$ is
induced by the lexicographic order $\succ_{\mathbf x}$  on $L$.
\end{proof}

\begin{proof}[Proof of Theorem \ref{geometry of
order 1}]  Theorem \ref{geometry of order 1} follows from
 Theorem \ref{convex hull and order}. Indeed, if a total order $\succ$  on $L$ is compatible with addition and with multiplication on positive number, then by Lemma \ref{lemma8} the set $L_+$ responsible of the order $\succ$ is convex. By  Theorem  \ref{convex hull and order} there is a coordinate system $\mathbf x$ on $L$ such that $L+\subset L_+(\mathbf x)$. Moreover from  the proof of Theorem \ref{convex hull and order} one can see that $L\cap  L_{+}(\mathbf x)=L_+$.
 \end{proof}

 \begin{proof}[Proof of Theorem  \ref{geometry of
order 2}]
One  can be prove Theorem  \ref{geometry of
order 2} in the same way as Theorem \ref{convex hull and order}.
Indeed, since the set $X_+$ is convex and does not contain the origin. Thus  by  version of Separation Theorem there
is a system of coordinates $\mathbf x$ such that $X_+\subset L_{+}(\mathbf x)$.

\medskip

Let us show that $ L_{+}(\mathbf x)\subset X_+$. Indeed, if there is a point $a\in  L_{+}(\mathbf x)$ not belonging to $X_+$, then $-a\in X_+$. But $-a$ do not belong to $L_{+}$. We obtain a contradiction which proves the needed statement.
\end{proof}

 \subsection{Orders compatible  with addition on subgroups~of~$\mathbb Q^n$}

 The following Theorem holds.

\smallskip

\begin{theorem}\label{th1} Let $A$ be any subset of the $n$-dimensional  vector  space  $\mathbb Q^n$  over the field $\mathbb Q$ of rational numbers. Then, the  semigroup $S_A$ generated by  the set $A$ contains the origin if and only if the convex hull $\Delta(A)$  of $A$ contains the origin.
\end{theorem}

\begin{proof} Assume that $0\in \Delta(A)$. By Caratheodory's  theorem there is a set $B\subset A$ of affinely independent points $\{a_1,\dots, a_{k+1} \}$ and a $(k+1)$-tuple of nonnegative numbers $\{\lambda_i\}$  such that $0=\lambda_ia_i$ and $\sum\lambda_i=1$.
\smallskip

Since the points  $a_i$and the origin  belong to $\mathbb Q^n$,  and  $a_1,\dots, a_k$ are affinely independent,  all numbers $\lambda_1,\dots,\lambda_k$ are rational.
\smallskip

 Multiplying the identity $\sum \lambda_i a_i$  by the product  of the denominators of rational numbers  $\lambda_i$, we obtain the relation $$\sum q_ia_i=0,$$ where $q_i$ are natural numbers.
 \medskip

This identity means that the semigroup $S_A$, together with the set $A$, contains the origin.
 \medskip

On the other hand, if $0\in S_A$, then there are points  $a_i\in A$ and natural numbers  $q_i$, such that $\sum q_ia_i=0$. Dividing this identity by $Q=\sum q_i$, and putting $\lambda_i=\dfrac{q_i}{Q}$, we obtain the representation of the origin in the form $0=\sum \lambda_i a_i$ where $\lambda_i>0$ and $\sum \lambda_i=1$. This means that the origin belongs to the convex hull of the set $A$.
\end{proof}

\begin{theorem}\label{result on Q^n} A total order $\succ$  on a subgroup  $G$ of the additive group of the $n$-dimensional vector space $ \mathbb Q^n\subset \mathbb R^n$ over rational numbers is  compatible with addition if and only if the order $\succ$ is induced from the some lexicographic order $\succ_{\mathbf x}$  on $\mathbb R^n$.
\end{theorem}

\begin{proof}Assume that the order $\succ$ on $G$ is compatible with addition.  The semigroup $G_+\subset G \subset \mathbb Q^n$ which is responsible   for the  order $\succ$ cannot contain the origin. Thus, by Theorem \ref{th1}, the convex hull $\Delta(G_+)$ of this semigroup also does not contain the origin. So, by Theorem \ref{convex hull and order}, the order $\succ$ is induced by some lexicographic order $\succ_\mathbf x$  on $\mathbb R^n$.

On the other hand, the order $\succ$ induced by any lexicographic order $\succ_\mathbf x$ on $\mathbb R^n$ is compatible with addition.
\end{proof}

\section{Well-ordered semigroups}

In this section we discuss well-ordered finitely generated  subsemigroups of the  additive group of   finite dimensional vector spaces over the real numbers and over the rational numbers.

\subsection{Well-ordered semigroups of $\mathbb R^n$}

The following theorem holds:

\begin{theorem}\label{th on well order 1} Assume that  an  order $\succ$ of a finitely generated semigroup
 $S\subset L$ of the  additive group of a  real $n$-dimensional space $L$
 is induced by the lexicographic order $\succ_{\mathbf x}$  related to some coordinate system $\mathbf x$ on $L$. Then, $S$ is a well-ordered set with respect to the order $\succ$ if and only if the ordered semigroup   $S$ satisfies the condition from Lemma \ref{well-ordered 1}.
\end{theorem}

\begin{proof} If $S$ is a well-ordered set, then $S$ satisfies  the conditions of Lemma \ref{well-ordered 1}. Let us prove the Theorem in the opposite direction.
\medskip

We will use induction on the dimension  $n$ of the ambient space $L$. Note that for $n=1$, the Theorem is already proven above (see Lemma \ref{lemma14}).
\medskip

Assume that the order on $S\subset L$, where $\dim L=n$, is induced by the lexicographic
order $\succ_\mathbf x$  on $L$  related  to the coordinates system
$\mathbf x=(x_1,\dots,x_n)$.
\medskip

Let $A$ be a set of generators of the semigroup $S$ and let $B$ be a subset of  $A$ on which the function $x_1$ is positive. Since $S$ satisfies the condition in  Lemma \ref{well-ordered 1}, the function $x$ is nonnegative on $S$. So, the function $x$ is equal to zero on the set
 $C=A\setminus B$.
\medskip

Denote by $S_{x_1}$ the semigroup $S\cap \{x_1=0\}$ lying in the space
$x_1=0$ of dimension $n-1$. The semigroup $S_{x_1}$ is generated by elements from the set $C$, and equipped with the lexicographic order related to the coordinate system $(x_2,\dots,x_n)$ on the hyperplane $x_1=0$. By induction, the semigroup  $S_{x_1}$  is well-ordered.

\medskip

Let us consider the image $x_1(S)$ of the semigroup  $S$ under the map $x_1:L\to \mathbb R$. The set $x_1(S)$ is an additive subsemigroup of $\mathbb R$ generated by elements from the set $x_1(B)$ (and by zero, if the set $S_{x_1}$ is not empty).
\medskip

The finitely generated semigroup $x_1(S)\subset \mathbb R$ is a well-ordered set by Lemma \ref{lemma14}.

\medskip

Let us show that any set $D\subset S$ contains a smallest element. First, there is a
smallest value $d=\min x_1(D)$  of  the function $x_1$ on the set $D$ since the set $x_1(s)$ is well-ordered.

\medskip

Consider a semigroup $S(B)$ generated by elements of the  set $B$. Let $F\subset S(B)$ be a subset on which the function $x_1$ is equal to $d$. The set $f$ is finite since the function $x_1$ is positive on the finite set $B$.
\medskip

The set $\{x_1=d\}\cap S\subset S$ on which $x_1$ is equal to $d$ is covered by a finite collection of shifted copies $S_{x_1}+f$ of the semigroup $S_{x_1}$, where $f$ is any element of $F$.

\smallskip

Each such copy $S_{x_1}+f$  is a well-ordered set, since $S_{x_1}$ is well-ordered and its order is compatible with addition. So, the set  $\{x_1=d\}\cap S$ is  well-ordered.
\medskip

We see that the set $D\subset S$ contains a smallest element: the smallest element of the set $\{x_1=d\}\cap D$, where $d$ is the smallest value of $x_1$ on $d$.
\end{proof}

\subsection{Well-ordered semigroups of~$\mathbb Q^n$}

\begin{theorem}\label{th on well order 2}
Assume that  an  order $\succ$ on a finitely generated semigroup
 $S\subset \mathbb Q^n$ of the  additive group of an  $n$-dimensional vector space $\mathbb Q^n$ over rational numbers is compatible with addition.
 Then, $S$ is a well-ordered set with respect to the order $\succ$ if and only if the ordered semigroup   $S$ satisfies the condition in Lemma \ref{well-ordered 1}.
\end{theorem}

\begin{proof} An order $\succ$ compatible with addition can be uniquely extended to an order compatible with addition to the group $G$, generated by the semigroup $S$.

By Theorem \ref{result on Q^n}, any order compatible with addition on the group $S\subset \mathbb Q^n$ is induced by a lexicographic order. To complete the proof it is enough to use Theorem \ref{th on well order 1}
\end{proof}

\section{ Subgroups and Subsemigroups of the lattice $\mathbb Z^n$}

Let us apply the  results discussed above to  additive subgroups and subsemigroups  of the standard lattice $\mathbb Z^n$.

The lattice $\mathbb Z^n$ is naturally embedded in the spaces $\mathbb Q^n \subset\mathbb R^n$.

\medskip

\begin{theorem}  Each total order on the group $\mathbb Z^n$ and on any of its subssemigroup $S\subset \mathbb Z^n$ that is compatible with addition is induced by a lexicographic order $\succ_\mathbf x$   on $\mathbb R^n$ that is related to some coordinate system $\mathbf x=(x_1,\dots,x_n)$.
\medskip

Such an order $\succ_\mathbf x$ on  a finitely generated semigroup $S=\mathbb Z^n$ is a well-order on $S$ if and only if any nonzero element $a\in S$ satisfies the inequality $2a\succ_{\mathbf x} a$.
\smallskip

In particular the semigroup $\mathbb Z^n_{\geq 0}$ consisting of integral points with nonnegative coordinates is a well-ordered set with respect to the lexicographic order $\succ_\mathbf X$ if and only if the origin is the smallest element of the semigroup $\mathbb Z^n_{\geq 0}$.

\end{theorem}

 \end{document}